\documentclass[11pt,a4paper]{article}

\usepackage[a4paper,margin=2.7cm]{geometry}
\usepackage[T1]{fontenc}
\usepackage[utf8]{inputenc}
\usepackage{lmodern}
\usepackage{microtype}
\usepackage{amsmath,amssymb,amsthm,mathtools}
\usepackage{xcolor}
\usepackage{hyperref}
\usepackage{comment}
\hypersetup{
  colorlinks=true,
  linkcolor=blue!45!black,
  citecolor=green!35!black,
  urlcolor=blue!55!black
}

\newtheorem{theorem}{Theorem}[section]
\newtheorem{lemma}[theorem]{Lemma}

\newtheorem{corollary}[theorem]{Corollary}
\theoremstyle{definition}
\newtheorem{definition}[theorem]{Definition}
\newtheorem{example}[theorem]{Example}
\theoremstyle{remark}
\newtheorem{remark}[theorem]{Remark}
\newtheorem{notation}[theorem]{Notation}

\newcommand{\F}{\mathbb{F}}
\newcommand{\supp}{\operatorname{supp}}
\newcommand{\sr}{\operatorname{sr}}

\title{On the Slice Rank of Tensors in \texorpdfstring{$P$}{P}-Echelon Form}
\author{Omran Ahmadi$^\dagger$ \and Hassan Norouzi$^*$}

\date{
\vspace*{3mm}
\parbox{\linewidth}{
\centering
\small
 $^\dagger$ School of Mathematics,\\ 
  Institute for Research in Fundamental Sciences (IPM),\\ 
  E-mail: \texttt{oahmadid@ipm.ir}
 \endgraf\medskip
  $^*$ School of Mathematics,\\ 
  Institute for Research in Fundamental Sciences (IPM),\\
  E-mail: \texttt{norouzi@ipm.ir}
  \endgraf\medskip
  \today
  }
}

\begin{document}
\maketitle

\begin{abstract}

For a totally ordered finite set $A$ with the total order $Q$, a poset $P=([d],\le_P)$, and a field $\F$, a tensor \(T:A^d\longrightarrow \F\) is in $P$-echelon form if $r\le_{P}s$ in $P$ implies $a_r\le_Q a_s$ in every tuple \((a_1,\dots,a_d)\) in the support of $T$. We prove that if the Hasse diagram of $P$ has no isolated vertex, then tensors in $P$-echelon form with nonzero diagonal entries have full slice-rank. Our results extend and improve on recent results of Amanov and Yeliussizov.  

\end{abstract}

\section{Introduction}
Let $A_i$, $i=1,\ldots,d$, be some finite sets and $\mathbb{F}$ be a field. A $d$-tensor $T$ is a function from $\prod_i A_i$ to $\mathbb{F}$. A tensor $T$ is called a {\it{slice}} if it can be decomposed as \(T = f g\) where $f$ is a 1-tensor from $A_j$ to $\mathbb{F}$ for some $j$ in $[d]:=\{1,2,\ldots,d\}$ and $g$ is a $(d-1)$-tensor from $\prod_{i\neq j}A_i$ to $\mathbb{F}$. The important notion of the slice-rank of a $d$-tensor is:
\begin{definition}[Slice-rank]
    Let $T\colon \prod A_i \longrightarrow \mathbb{F}$ be a $d$-tensor. The slice-rank denoted $\operatorname{sr}(T)$ of $T$ is the minimum $r$ such that $T$ can be written as the sum of $r$ slices, i.e., $\operatorname{sr}(T)$ is the minimum number $r$ such that
    \[
        T= \sum_{i=1}^r f_i g_i
    \]
    where \(f_ig_i\) is a slice for each \(i\).
\end{definition}

When $A_1=\cdots=A_d$, a $d$-tensor $T$ is diagonal if
\(T(a_1,\ldots,a_d)\neq 0\)
implies 
\(a_1=\cdots=a_d\). Tao's slice-rank lemma for diagonal tensors is: 
\begin{lemma}[T. Tao \cite{TaoWeblog}]\label{diagonalTensor}
    \label{slicerank}
	Let $A$ be a finite set and $\mathbb{F}$ be a field. If $T\colon A^d \rightarrow \mathbb{F}$ is a diagonal $d$-tensor with nonzero diagonal entries, then $\operatorname{sr}(T)=|A|$.
\end{lemma}

Tao's slice-rank lemma which is rooted in a breakthrough paper by Lev, Croot and Pach~\cite{croot_progression-free_2016} and a subsequent paper of Ellenberg and Gijswijt~\cite{ellenberg2016large} has been at the heart of the slice-rank method in extremal combinatorics. 

There has been attempts to generalize the above lemma to other types of tensors~\cite{AhmadiNorouzi, AhmadiNorouzi2026, AmanovYeliussizov2023}. For example, the generalization for even $d$ to the $d$-tensors in the so-called $P$-echelon form with nonzero diagonal entries by Amanov and Yeliussizov \cite{AmanovYeliussizov2023} where for a totally ordered finite set $A$ with the total order $Q$ and a poset $P=([d],\le_P)$ with connected Hasse diagram, a tensor \(T:A^d\longrightarrow \F\) is in $P$-echelon form if \(T(a_1,\dots,a_d)\ne 0\) implies that if $r\le_{P}s$, then $a_r\le_Q a_s$.  

In this paper, extending the result of Amanov and Yeliussizov to posets whose Hasse diagram has no isolated vertex we prove that:  

\begin{theorem}[Main theorem]\label{thm:main-intro}
Let $d\ge 2$. Furthermore, let $A$ be a totally ordered finite set with the total order $Q$, $\F$ be a field, and let $P=([d],\le_P)$ be a poset whose Hasse diagram has no isolated vertex. Suppose \(T:A^d\longrightarrow \F\)
is in $P$-echelon form with nonzero diagonal entries, i.e., $T(a,\ldots,a)\neq 0$ for every $a\in A$. 
Then
\[
\sr(T)=|A|.
\]
In particular, the conclusion holds for both even and odd $d$.
\end{theorem}
\begin{remark}
    Amanov and Yeliussizov proved that a tensor with nonzero diagonal entries in $P$-echelon form where Hasse diagram of $P$ is connected has full rank when $d$ is even. Our improvement is twofold. First, we extend their theorem to include tensors of odd order. Second, we drop the assumption that the Hasse diagram of $P$ is connected and replace it with the assumption that the Hasse diagram of $P$ has no isolated vertex.
\end{remark}

\section{Proof of theorem~\ref{thm:main-intro}}

Our tool to prove the main result is the following functional reformulation of Proposition 4 which has appeared in~\cite{SawinTaoNotes}.

\begin{lemma}[Sawin-Tao \cite{SawinTaoNotes}]
\label{tao-cover-lemma}
Let $(A_i,\le_i)$ for $i=1,\ldots,d$ be totally ordered finite sets, and $\preceq$ be the partial order on $\prod_{i=1}^dA_i$ induced from the orders $\le_i$ on $A_i$, i.e., $(a_1,\ldots,a_d)\preceq (b_1,\ldots,b_d)$ if $a_i\le_i b_i$ for every $i$. Furthermore, let $\mathbb{F}$ be a field, $T:\prod_{i=1}^dA_i\to \mathbb{F}$ be a $d$-tensor, 
\[
\operatorname{supp}(T)=\left\{x\in \prod_{i=1}^dA_i : T(x)\neq 0\right\}, 
\]
and $\Gamma\subseteq \operatorname{supp}(T)$ be the set of maximal elements of $\operatorname{supp}(T)$. Then
\[
\operatorname{sr}(T)\ge\operatorname{\min_{\substack{\Gamma_1,\cdots,\Gamma_d\\\Gamma=\bigcup_{j=1}^d\Gamma_j}}}\sum_{j=1}^d|\pi_j(\Gamma_j)|
\]
where the minimum is taken over all $\Gamma_1,\ldots,\Gamma_d\subseteq\prod_{i=1}^dA_i $ such that 
\[
\Gamma=\bigcup_{j=1}^d\Gamma_j,
\]
and $\pi_j(\cdot)$ denotes the projection onto the $j$-th coordinate. 
\end{lemma}

For the sake of simpler exposition, we have the following definition. 
\begin{definition}
Let $d\ge 2$. Furthermore, let $A$ be a totally ordered finite set with the total order $Q$, and let $P=([d],\le_P)$ be a poset. We say $x=(x_1,\dots,x_d)$ is in $P$-echelon form if $r\le_{P}s$ implies $x_r\le_Q x_s$.
\end{definition}

Notice that with the above definition, a $d$-tensor $T:A^d\longrightarrow \mathbb{F}$ is in $P$-echelon form if all the elements of $\operatorname{supp}(T)$ are in $P$-echelon form. 

\begin{notation}
 For $a\in A$, let 
\[
\delta_a=(a,\dots,a)\in A^d, 
\]   
and 
\[
\operatorname{Diag}(A^d):=\{\delta_a:a\in A\}.
\]
\end{notation}

\begin{lemma}\label{lem:tree-orders}
Let $P=([d],\le_P)$ where $d\ge 2$ be a poset with whose Hasse diagram $H$ has no isolated vertex, and let $A$ be a totally ordered finite set with the order $\le_Q$. Then, for each coordinate $j\in[d]$, it is possible to assign either the original order $\le_Q$ or the reverse order $\ge_Q$ on $A$ so that if $x=(x_1,\dots,x_d)$ is in $P$-echelon form and \(\delta_a\preceq x\) in the resulting product order \(\preceq\) on $A^d$, then $x=\delta_a$.
\end{lemma}

\begin{proof}
Let $M$ and $N$ denote the set of maximal and non-maximal elements of the coordinate poset $P$, respectively. We assign $\ge_Q$ and $\le_Q$ to each coordinate $m\in M$ and $n\in N$, respectively. Since $H$ has no isolated vertex, the set of non-maximal elements $N$ is non-empty, and furthermore, for every $m\in M$ there exists $n\in N$ such that $n\le_P m$. Now, let $\delta_a\preceq x$ in the resulting product order, $m$ be an arbitrary element of $M$ and $n$ be an element of $N$ such that $n\le_P m$.  Since $x$ is in $P$-echelon form, $\delta_a\preceq x$ and $n\le_P m$, we have $x_n\le_Q x_m$, $a\ge_Q x_m, a\le_Q x_n$. Thus $x_m=x_n=a$. The same argument holds if take $n$ to be an arbitrary element of $N$ and $m$ be an element in $M$ such that $n\le_P m$. Hence the claim holds.   
\end{proof}

We can now prove the main result.

\begin{theorem}\label{thm:main}
Let $P=([d],\le_P)$ where $d\ge 2$ be a poset whose Hasse diagram $H$ has no isolated vertex, and let $A$ be a totally ordered finite set with the order $\le_Q$. If $T$ is a $d$-tensor in $P$-echelon form from $A^d$ to a field $\F$ with nonzero diagonal entries, then 
\[
\sr(T)=|A|.
\]
\end{theorem}

\begin{proof}
As the diagonal elements of $T$ are nonzero,
\(
\operatorname{Diag}(A^d)
\)
is a subset of $\operatorname{supp}(T)$. Furthermore, if we let $\preceq$ be the order on $A^d$ from Lemma~\ref{lem:tree-orders}, and let $\Gamma$ denote the set of maximal elements of $\supp(T)$, then Lemma~\ref{lem:tree-orders} implies that $\operatorname{Diag}(A^d)\subseteq \Gamma$. 

Now, if \(
\Gamma=\Gamma_1\cup\cdots\cup\Gamma_d,
\)
then each $\delta_a$ belongs to some $\Gamma_j$, and therefore
$a\in\pi_j(\Gamma_j)$. Consequently,
\[
A\subseteq\bigcup_{j=1}^d\pi_j(\Gamma_j)
\]
and hence
\[
\sum_{j=1}^d|\pi_j(\Gamma_j)|\geq |A|.
\]
Thus, from Lemma~\ref{tao-cover-lemma} we have
\[
\operatorname{sr}(T)\geq |A|.
\]
The reverse inequality is immediate from the trivial slicing
\[
T(a_1,\ldots,a_d)
 =
 \sum_{a\in A}
 \mathbf{1}_{\{a_1=a\}}\,
 T(a,a_2,\ldots,a_d).
\]
Hence
\[
\operatorname{sr}(T)=|A|.
\qedhere
\]
\end{proof}

\end{document}